\documentclass[12pt]{article}
\usepackage{amssymb}
\usepackage{mathrsfs}
\usepackage[hypertex]{hyperref}
\usepackage{amsfonts}
\usepackage{graphicx}
\usepackage{mathptmx}
\usepackage{latexsym,amsmath,amssymb,amsfonts,amsthm}

\newtheorem{theorem}{Theorem}[section]
\newtheorem{lemma}[theorem]{Lemma}

\newtheorem{remark}[theorem]{Remark}

\begin{document}
\setcounter{page}{1}
\title{Monotonicity of the first eigenvalue of \\
the Laplace and the $p$-Laplace operators under a forced mean
curvature flow}
\author{Jing Mao}
\date{}
\protect\footnotetext{\!\!\!\!\!\!\!\!\!\!\!\! {MSC 2010: 58C40;
53C44}
  }
\maketitle ~~~\\[-15mm]
\begin{center}{\footnotesize  Department of Mathematics, Harbin Institute
of Technology (Weihai), Weihai, 264209, China \\
jiner120@163.com, jiner120@tom.com}
\end{center}

%\\[5mm]
\begin{abstract}
In this paper, we would like to give an answer to \textbf{Problem 1}
below issued firstly in [J. Mao, Eigenvalue estimation and some
 results on finite topological type, Ph.D. thesis, IST-UTL, 2013].
 In fact, by imposing some conditions on the mean curvature of the initial
 hypersurface and the coefficient function of the forcing term of a forced
 mean curvature flow considered here, we can obtain that
 the first eigenvalues of the Laplace
 and the $p$-Laplace operators are monotonic under this flow. Surprisingly, during this process, we
 get an interesting byproduct, that is, without any complicate
 constraint, we can give lower bounds for the
 first nonzero closed eigenvalue of the Laplacian provided additionally the second fundamental form of the initial hypersurface
 satisfies a pinching condition.
\end{abstract}

\markright{\sl\hfill  J. Mao \hfill}

\section{Introduction }
\renewcommand{\thesection}{\arabic{section}}
\renewcommand{\theequation}{\thesection.\arabic{equation}}
\setcounter{equation}{0} \setcounter{maintheorem}{0}

The mathematical genius, Perelman, in his famous work \cite{gpere}
introduced a functional, which is called $\mathcal{F}$-functional,
for a prescribed closed Riemannian manifold $(M,g)$ and a function
$f$ on $M$ defined as follows
\begin{eqnarray*}
\mathcal {F}(g,f):=\int_{M}\left(R+|\nabla{f}|^{2}\right)e^{-f}d\mu,
\end{eqnarray*}
with $R$ here the scalar curvature and $d\mu$ the volume element of
$M$. Denote by $\nabla$ and $\Delta$ the gradient and the Laplace
operators of $M$, respectively. For the following coupled system
\begin{eqnarray*}
\left\{
\begin{array}{ll}
\frac{\partial}{\partial{t}}g_{ij}=-2R_{ij},\\
\frac{\partial}{\partial{t}}f=-\Delta{f}-R+|\nabla{f}|^{2},
\end{array}
\right.
\end{eqnarray*}
with the first equation the famous Ricci-Hamilton flow,  he proved
that the $\mathcal{F}$-functional is nondecreasing under the Ricci
flow, i.e.
\begin{eqnarray*}
\frac{d}{dt}\mathcal{F}=2\int_{M}\left|R_{ij}+\nabla_{i}\nabla_{j}f\right|^{2}e^{-f}d\mu\geq0.
\end{eqnarray*}
Define
\begin{eqnarray*}
\lambda(g):=\inf\left\{\mathcal{F}(g,f)\Big{|}f~{\rm{runs~over~all~smooth~functions,
and ~satisfies}}\int_{M}e^{-f}d\mu=1\right\},
\end{eqnarray*}
and then $\lambda(g)$ is the lowest eigenvalue of the operator
$(-4\Delta+R)$. This fact can be obtained easily by making a
transformation $u=e^{-f/2}$. Then $\lambda(g)$ can be defined
equivalently as follows
\begin{eqnarray*}
\lambda(g):=\inf\left\{\int_{M}(4|\nabla{u}|^{2}+Ru^{2})d\mu\Big{|}u~{\rm{runs~over~all~smooth~functions,~and}}
\int_{M}u^{2}d\mu=1\right\},
\end{eqnarray*}
which implies that $\lambda(g)=\lambda_{1}(-4\Delta+R)$, the first
eigenvalue of $(-4\Delta+R)$. Besides, $\lambda(g)$ is nondecreasing
since $\mathcal{F}$ is nondecreasing. By using this fact, Perelman
has shown that there are no nontrivial steady or expanding breathers
on compact manifolds (see sections 2, 3, and 4 of \cite{gpere}).

From Perelman's this work, we know that monotonicity of the first
eigenvalue of some operator related to the Laplacian under curvature
flows, like the Ricci flow, should be worthy to be investigated.
Because of this, many mathematicians have made efforts on this
direction, and some interesting results have also been obtained
after Perelman's pioneering work. For instance, Ma \cite{lma}
studied the first eigenvalue of the Laplace operator $\Delta$,
subject to the Dirichlet boundary condition, on a compact domain,
with smooth boundary in a compact or a complete noncompact manifold,
under the unnormalized Ricci-Hamilton flow, and obtained the
monotonicity of the first eigenvalue of $\Delta$ under several
assumptions on the scalar curvature of the prescribed manifold
therein. Cao \cite{caox} showed that, under the Ricci flow, the
eigenvalues of the operator $(-\Delta+R/2)$, with $R$ the scalar
curvature, are non-decreasing for manifolds with nonnegative
curvature operator, and then, by applying this monotonicity of the
eigenvalues, he proved that the only steady Ricci breather with
nonnegative curvature operator is the trivial one (see section 4 of
\cite{caox}). Without assuming the nonnegativity of the curvature
operator, Li \cite{jfli} also proved the nondecreasing property for
the eigenvalues of the operator $(-\Delta+R/2)$. Cao \cite{caox2}
proved that, under the unnormalized Ricci flow, the first eigenvalue
of $(-\Delta+cR)$, with $c\geq1/4$ and $R$ the scalar curvature, is
nondecreasing, which generalized his previous work \cite{caox}.
Recently, Cao, Hou, and Ling \cite{xsj} derived a monotonicity
formula for the first eigenvalue of the operator $(-\Delta+aR)$,
with $0<a\leq1/2$, on closed surfaces with the scalar curvature
$R\geq0$ under the unnormalized Ricci flow.

The mean curvature flow (MCF) also has connections with the Ricci
flow which is a powerful tool to solve the $3$-dimensional
Poincar\'{e} conjecture. There are surprising analogies between the
Ricci flow and the MCF. Indeed, many results hold in a similar way
for both flows, and several ideas have been successfully transferred
from one context to the other (see, for instance, \cite[Corollary
2.5]{lmw}, where we have used a principle, the maximum principle for
tensors, appearing in the Ricci flow, supplied by Hamilton, to prove
the convexity-preserving property for the curvature flow considered
therein). However, at the moment there is no formal way of
transforming one of them into the other.

Because of the deep connection between the MCF and the Ricci flow,
it is natural to ask whether or not we could derive monotonicity
formulas for the first eigenvalue of some geometric operators
related to the Laplacian under the MCF or some other deformations of
the MCF, like the volume-preserving MCF, the area-preserving MCF,
the forced MCF (MCF with a prescribed forcing term), etc. Recently,
under several assumptions on the mean curvature of a given closed
Riemannian manifold, Zhao \cite{zhaol} proved that the first
eigenvalue of the $p$-Laplacian on the manifold is nondecreasing
along powers of the $m$th MCF (see, e.g., \cite{eccs} for the basic
information on this flow). This provides us the feasibility of
trying to derive the monotonicity of the first eigenvalue of the
Laplacian or the $p$-Laplacian under curvature flows.

Denote by $M_{0}$ a compact and strictly convex hypersurface of
dimension $n\geq2$, without boundary, smoothly embedded in the
Euclidean space $\mathbb{R}^{n+1}$ and represented locally by a
diffeomorphism
$X_{0}:U\subset\mathbb{R}^{n}\rightarrow{X_{0}(U)}\subset{M_{0}}\subset\mathbb{R}^{n+1}$.
Consider that $M_{0}$ evolves along the forced MCF defined as
follows
\begin{eqnarray}  \label{1.1}
\left\{
\begin{array}{ll}
\frac{\partial}{\partial{t}}X(x,t)=-H(x,t)\vec{v}(x,t)+\kappa(t)X(x,t),
\qquad
x\in{M}^{n}_{0}, \quad t>0,\\
X(\cdot,0)=X_{0}, & \quad
\end{array}
\right.
\end{eqnarray}
with $\vec{v}(x, t)$ the outer unit normal vector of
$M_{t}=X_{t}(M_{0})$ at $X(x, t)=X_{t}(x)$, $H$ the mean curvature
of $M_{t}$, and $\kappa(t)$ a continuous function of $t$. Li, Mao
and Wu \cite{lmw} proved that the convexity is preserving as the
case of MCF, and the evolving convex hypersurfaces may shrink to a
point in finite time if the forcing term is small, or exist for all
time and expand to infinity if it is large enough (see \cite[Theorem
1.1]{lmw} or Theorem \ref{theoremmain} here for the precise
statement). In fact, the forced MCF (\ref{1.1}) can be obtained by
adding a forcing term in direction of the position vector to the
classical MCF (only when the ambient space is a Euclidean space),
and this type of forced (or forced hyperbolic) mean curvature flows
has been studied in \cite{lmw,m2,m3,mlw} with some interesting
results on the convergence or the long time existence obtained.

As pointed out in \cite{lmw}, the tangent component of $X(x,t)$ does
not affect the behavior of the evolving hypersurface, but usually
the normal component of $X(x,t)$ is not a unit normal vector, which
leads to the fact that the flow (\ref{1.1}) differs from the
classical MCF. Readers can find that the convergent situation of our
flow (\ref{1.1}) is more complicated than that of the MCF even if
the initial hypersurface is a sphere (see Remark \ref{remarkadd1}).
In fact, it can be seen as an extension of the MCF, since the flow
(\ref{1.1}) degenerates to be the MCF if $\kappa(t)\equiv0$.

Based on the result concerning the convergence or the long time
existence we have obtained in \cite{lmw}, and the fact that Zhao can
get a monotonicity formula for the first eigenvalue of the
$p$-Laplacian under powers of the $m$th MCF in \cite{zhaol}, we
might consider the following problem.

 \vskip 1 mm
{\bf Problem 1}. \emph{ For a compact and strictly convex
hypersurface $M_{0}$ of dimension $n\geq2$, without boundary, which
is embedded smoothly in $\mathbb{R}^{n+1}$ and can be represented
locally by a diffeomorphism
$X_{0}:U\subset\mathbb{R}^{n}\rightarrow{X_{0}(U)}\subset{M_{0}}\subset\mathbb{R}^{n+1}$,
could we derive a monotonicity formula for the first eigenvalue of
the Laplace and the $p$-Laplace operators on $M_{t}$ under the
forced MCF defined by (\ref{1.1})? }

Several eigenvalue problems have been studied by the author in
\cite{fmi,m,m1,PhDJingMao} and some interesting conclusions have
been obtained therein. This experience somehow supplies the
possibility to answer the above \emph{Problem 1}. In fact, based on
the main conclusions for the flow (\ref{1.1}) in \cite{lmw}, we can
give an answer to this problem (see Theorem \ref{theorem3} for the
details).

As mentioned in the Abstract, during the process of trying to get
the monotonicity of the first non-zero closed eigenvalue, we can
obtain an interesting byproduct, which somehow reveals the
convergence or expansion of the evolving hypersurfaces under the
flow (\ref{1.1}) from the aspect of eigenvalues. As in Section 2,
denote by $H$ the mean curvature, $h_{ij}$ and $g_{ij}$ the
components of the second fundamental form and the Riemannian metric
of the prescribed manifold, respectively. By imposing a pinching
condition for the second fundamental form of the initial
hypersurface, we can prove the following.

\begin{theorem} \label{theorembound}
If, in addition, there exist positive constants
$\alpha_{1},\alpha_{2}, \ldots,\alpha_{n}$ such that the initial
hypersurface $M_0$ satisfies
\begin{eqnarray}  \label{pinchingc}
h_{ij}=\alpha_{i}Hg_{ij}, \qquad
\mathrm{where}~\sum\limits_{i=1}^{n}\alpha_{i}=1~\mathrm{and}
~\left|\alpha_{i}-\frac{1}{n}\right|\leq\epsilon
\end{eqnarray}
for small enough $\epsilon$ only depending on $n$,
 then under the flow (\ref{1.1}) we have
\begin{eqnarray*}
\lambda_{1}(t)\geq
e^{-2\int_{0}^{t}\kappa(\tau)d\tau}\cdot\lambda_{1}(0)
\end{eqnarray*}
for any $0\leq{t}<T_{\mathrm{m}}$, where, of course,
$\lambda_{1}(0)$ and $\lambda_{1}(t)$ are the first nonzero closed
eigenvalues of the Laplace operator on $M_0$ and $M_t$ respectively,
and $T_{\mathrm{m}}$ is defined by (\ref{key}).
\end{theorem}

\begin{remark}
\rm{For an $n$-dimensional compact, connected and oriented
Riemannian manifold $(M,g)$ without boundary isometrically immersed
in $\mathbb{R}^{n+1}$, it is said to be \emph{almost-umbilical} if
there exists $\theta\in(0,1)$ such that
$\|A-cg\|_{\infty}\leq\epsilon$ for a positive constant $c$, with
$\epsilon$ small enough depending on $n$, $c$ and $\theta$, where
$A$ is the second fundamental form of $M$. So, clearly, if the
initial hypersurface $M_{0}$ satisfies the pinching condition
(\ref{pinchingc}), then it is almost-umbilical. A well-known result
states that a totally umbilical hypersurface of $\mathbb{R}^{n+1}$
which is not totally geodesic is a round sphere. Clearly a totally
umbilical hypersurface of $\mathbb{R}^{n+1}$ must be
almost-umbilical with $c=H/n$. However, an almost-umbilical
hypersurface of $\mathbb{R}^{n+1}$ may not be totally umbilical. For
instance, considering a sphere with ideal elasticity in
$\mathbb{R}^3$, and orthogonally and very slightly squashing this
sphere at a pair of antipodal points such that the new geometric
object (might be an ellipsoid) obtained by this deformation
satisfies the almost-umbilical condition. In this case, the
deformation of the sphere might be ignored but it do has
deformation. Therefore, it is natural to ask if and how the
almost-umbilical hypersurfaces are ``\emph{close}" to round spheres.
In fact, there are many interesting conclusions walking on this
direction. For instance, Shiohama and Xu \cite{kx1,kx2} proved that
almost-umbilical hypersurfaces of Euclidean space are homeomorphic
to the sphere if imposing a condition on Betti numbers. Recently,
Roth \cite{jr} proved that an $n$-dimensional compact, connected and
oriented \emph{almost-umbilical} Riemannian manifold $M$ without
boundary isometrically immersed in $\mathbb{R}^{n+1}$ is
diffeomorphic and $\theta$-quasi-isometric to
$\mathbb{S}^{n}(\frac{1}{c})$, i.e. there exists a diffeomorphism
$F$ from $M$ into $\mathbb{S}^{n}(\frac{1}{c})$ such that, for any
$x\in{M}$ and any unitary vector $X\in{T_{x}M}$, we have
$\left||d_{x}F(X)|^{2}-1\right|\leq\theta$. Hence, according to
these facts, our pinching condition (\ref{pinchingc}) is feasible
and also reasonable. Especially, for (\ref{pinchingc}), when
$\alpha_{i}=1/n$ for each $1\leq{i}\leq{n}$, then the initial
hypersurface $M_{0}$ must be a sphere with a prescribed radius, say
$r_{0}$, and moreover, the evolving hypersurface $M_{t}$ must be a
sphere with radius $r(t)$ given by (\ref{solution1}) (see Remark
\ref{remarkadd1} for details). Correspondingly,
$\lambda_{1}(t)=n/r^{2}(t)$, which clearly satisfies the conclusion
of Theorem \ref{theorembound}.

}
\end{remark}

The paper is organized as follows. We recall some basic knowledge
about the Laplacian and the $p$-Laplacian in the next section.
Besides, we also mention some useful conclusions of the forced MCF
(\ref{1.1}). In Section 3, we give the proofs of Theorems
\ref{theorem1} and \ref{theorem2}. In Section 4, by applying Theorem
\ref{theorem1}, we successfully give lower bounds for the first
nonzero closed eigenvalue of the Laplace operator provided, in
addition, the initial hypersurface satisfies the pinching condition
(\ref{pinchingc}). Theorem \ref{theorem3} will be proved in the last
section.

\section{Preliminaries}
\renewcommand{\thesection}{\arabic{section}}
\renewcommand{\theequation}{\thesection.\arabic{equation}}
\setcounter{equation}{0} \setcounter{maintheorem}{0}

In this section, we would like to give a brief introduction to the
eigenvalue problem first and then recall some facts about the forced
MCF (\ref{1.1}).

In fact, due to the related conditions, the eigenvalue problem can
be classified into several types, but here we just focus on the
closed eigenvalue problem. For the consistency of the symbols, as
before, let $M_{0}$ be an $n$-dimensional compact Riemannian
manifold without boundary. The so-called closed eigenvalue problem
is actually to find all possible real $\lambda$ such that there
exists non-trivial functions $u$ satisfying
\begin{eqnarray*}
\Delta{u}+\lambda{u}=0, \qquad \mathrm{on}~M_{0}
\end{eqnarray*}
with $\Delta$ the Laplacian on $M_0$, which is given by
\begin{eqnarray*}
\Delta{u}=\mathrm{div}(\nabla{u})=\frac{1}{\sqrt{\det(g_{ij})}}\sum\limits_{i,j=1}^{n}\frac{\partial}{\partial{x_{i}}}\left(\sqrt{\det(g_{ij})}g^{ij}
\frac{\partial{u}}{\partial{x_{j}}}\right)
\end{eqnarray*}
in a local coordinate system $\{x_{1},x_{2},\ldots,x_{n}\}$ of
$M_{0}$. Here $\mathrm{div}$ and $\nabla$ denote the divergence
operator and the gradient operator on $M_{0}$, respectively.
Moreover,
$|\nabla{u}|^2=|\nabla{u}|^2_{g}=\sum_{i,j=1}^{n}g^{ij}\frac{\partial{u}}{\partial{x_{i}}}\frac{\partial{u}}{\partial{x_{j}}}$,
and $(g^{ij})=(g_{ij})^{-1}$ is the inverse of the metric matrix. It
is well-known that $\Delta$ only has discrete spectrum in this
setting ($M_{0}$ is compact without boundary). Each element in the
discrete spectrum is called the eigenvalue of the Laplacian
$\Delta$. It is easy to find that $0$ is an eigenvalue of $\Delta$
and whose eigenfunction should be chosen to be a constant function.
By Rayleigh's theorem and Max-min principle, together with the fact
that eigenfunctions belonging to different eigenvalues are
orthogonal, we know that the first non-zero (i.e. the lowest
non-zero) closed eigenvalue $\lambda_{1}(M)$ ($\lambda_{1}$ for
short) can be characterized by
 \begin{eqnarray}  \label{2.1}
 \lambda_{1}=\inf\left\{\frac{\int_{M_0}|\nabla{u}|^{2}d\mu_0}{\int_{M_0}|u|^{2}d\mu_0}\Big{|}
u\neq0,
u\in{W}^{1,2}(M_0),~\mathrm{and}~\int_{M_0}ud\mu_0=0\right\},
 \end{eqnarray}
where $W^{1,2}(M_0)$ is the completion of the set $C^{\infty}(M_0)$
of the smooth functions on $M_0$ under the Sobolev norm
\begin{eqnarray*}
\|u\|_{1,2}:=\left(\int_{M_0}|u|^{2}d\mu_0+\int_{M_0}|\nabla{u}|^{2}d\mu_0\right)^{1/2},
\end{eqnarray*}
 and  $d\mu_0$ denotes the volume element of $M_0$.

\emph{Now, we would like to make an agreement. That is, for the
convenience, in the sequel we will drop the volume element for each
integration appearing below. We also make an agreement on the range
of indices as follows}
\begin{eqnarray*}
1\leq{i,j,\ldots}\leq{n}.
\end{eqnarray*}

The $p$-Laplacian ($1<p<\infty$) is a natural generalization of the
Laplace operator. In fact, the so-called $p$-Laplacian eigenvalue
problem is to consider the following nonlinear second-order partial
differential equation (PDE for short)
 \begin{eqnarray*}
\Delta_{p}u+\lambda|u|^{p-2}u=0, \qquad \mathrm{on}~M_{0},
 \end{eqnarray*}
where, in local coordinates $\{x_{1},\ldots,x_{n}\}$ on $M_0$,
$\Delta_{p}$ is defined by
\begin{eqnarray*}
\Delta_{p}u=\frac{1}{\sqrt{\det(g_{ij})}}\sum\limits_{i,j=1}^{n}\frac{\partial}{\partial{x_{i}}}\left(\sqrt{\det(g_{ij})}g^{ij}|\nabla{u}|^{p-2}
\frac{\partial{u}}{\partial{x_{j}}}\right).
\end{eqnarray*}
Similar to the case of the linear Laplace operator, $\Delta_{p}$ has
discrete spectrum on $M_0$ when $M_0$ is compact. However, we do not
know whether it only has the discrete spectrum or not. This
situation is different from the case of the Laplacian, when the
domain considered is bounded. Besides, the first non-zero closed
eigenvalue $\lambda_{1,p}(M_0)$ ($\lambda_{1,p}$ for short) of
$\Delta_{p}$ can be characterized by
\begin{eqnarray} \label{2.2}
\lambda_{1,p}=\inf\left\{ \frac{\int_{M_0}|\nabla
u|^{p}}{\int_{M_0}|u|^{p}}\Big{|} u \in W^{1,p}(M_0), u\neq 0,
~\mathrm{and}~\int_{M_0}|u|^{p-2}u=0\right\},
\end{eqnarray}
with $W^{1,p}(M_0)$  the completion of the set $C^{\infty}(M_0)$
under the Sobolev norm
\begin{eqnarray*}
\|u\|_{1,p}:=\left(\int_{M_0}|u|^{p}+\int_{M_0}|\nabla{u}|^{p}\right)^{1/p}.
\end{eqnarray*}
Now, we would like to recall several evolution equations derived in
\cite{lmw}, which will be used to prove our main conclusions. In
fact, for the unnormalized forced MCF (\ref{1.1}), we have (cf.
\cite[Lemma 2.2]{lmw})
\begin{eqnarray} \label{2.3}
\frac{\partial}{\partial{t}}g_{ij}=-2Hh_{ij}+2\kappa(t)g_{ij}
\end{eqnarray}
and
 \begin{eqnarray}  \label{2.4}
  \frac{\partial{h_{ij}}}{\partial{t}}=\Delta{h_{ij}}-2Hh_{il}g^{lm}h_{mj}+|A|^{2}h_{ij}+\kappa(t)h_{ij}, \qquad\quad \frac{\partial{H}}{\partial{t}}=\Delta{H}+|A|^{2}H-\kappa(t)H,
 \end{eqnarray}
with $g_{ij}$ the component of the Riemannian metric on $M_t$, $H$
the mean curvature and $h_{ij}$, $|A|^{2}$ the component and the
squared norm of the second fundamental form of $M_t$, respectively.
Denote by $T_{\mathrm{max}}$ the maximal existence time of the
forced MCF (\ref{1.1}). In fact, the existence of $T_{\max}>0$ can
be obtained by the fact that the flow (\ref{1.1}) is a parabolic
equation and which can be converted to a second-order strictly
parabolic PDE, leading to the existence of the maximal time interval
$[0,T_{\mathrm{max}})$ (see, for instance, \cite{m2} for a detailed
explanation of this kind of trick). In order to know more
information about the flow (\ref{1.1}) as
$t\rightarrow{T_{\mathrm{max}}}$, as the case of the classical MCF,
we have to make a rescale to this flow. More precisely, for any
$t\in[0,T_{\mathrm{max}})$, let $\phi(t)$ be a positive factor such
that the hypersurface $\widetilde{M}_{t}$ defined by
$\widetilde{X}(x,t)=\phi(t)X(x,t)$ has total area equal to $|M_0|$
(i.e. the area of $M_0$). That is to say,
$\int_{\widetilde{M}_{t}}=|M_0|$. Differentiating this equality with
respect to $t$, we have
\begin{eqnarray}  \label{2.5}
\phi^{-1}\frac{\partial\phi}{\partial{t}}=\frac{1}{n}\frac{\int_{M_t}H^{2}}{\int_{M_t}}-\kappa(t)=\frac{1}{n}h-\kappa(t).
\end{eqnarray}
 At the same time, choosing a new time
variable
\begin{eqnarray*}
\widetilde{t}(t)=\int_{0}^{t}\phi^{2}(\tau)d\tau=\int_{0}^{t}\phi^{2}(\tau),
\end{eqnarray*}
then we have
\begin{eqnarray*}
\widetilde{g}_{ij}=\phi^{2}g_{ij}, \quad
\widetilde{h}_{ij}=\phi{h_{ij}}, \quad \widetilde{H}=\phi^{-1}H,
\quad |\widetilde{A}|^{2}=\phi^{-2}|A|^{2},
\end{eqnarray*}
and the evolution equation (\ref{1.1}) becomes
\begin{eqnarray}  \label{2.6}
\left\{
\begin{array}{lll}
\frac{\partial}{\partial{\widetilde{t}}}\widetilde{X}(x,t)=-\widetilde{H}\cdot\widetilde{\vec{v}}+\frac{1}{n}\widetilde{h}\widetilde{X},\\
\\
\widetilde{X}(\cdot,0)=\widetilde{X}_{0}, & \quad
\end{array}
\right.
\end{eqnarray}
where
$\widetilde{h}=\phi^{-2}h=\int_{\widetilde{M_{\widetilde{t}}}}\widetilde{H}^{2}/\int_{\widetilde{M_{\widetilde{t}}}}$.
Clearly, we can obtain the normalized evolution equation for the
metric as follows
\begin{eqnarray} \label{2.7}
\frac{\partial\widetilde{g}_{ij}}{\partial\tilde{t}}=\frac{\partial{t}}{\partial\tilde{t}}\frac{\partial(\phi^{2}g_{ij})}{\partial{t}}
=\frac{2}{n}\widetilde{h}\widetilde{g}_{ij}-2\widetilde{H}\widetilde{h}_{ij}.
\end{eqnarray}
 By \cite{lmw}, we know there always exists a time sequence
$\{T_{i}\}$ in $[0,T_{\mathrm{max}})$ such that
$T_{i}\rightarrow{T_{\mathrm{max}}}$ as $i\rightarrow\infty$, and
moreover the limit
\begin{eqnarray}  \label{2.8}
\lim_{T_{i}\rightarrow{T_{\mathrm{max}}}}\phi(T_i)=\Xi
\end{eqnarray}
 holds (see the end of Section 4 of \cite{lmw} for the detailed statement).
 About the forced MCF (\ref{1.1}) and its normalized flow
 (\ref{2.6}), Li, Mao and Wu proved the following
 conclusion (cf. \cite[Theorem 1.1]{lmw}).
 \begin{theorem} \label{theoremmain}
Let $M_0$ be an $n$-dimensional smooth, compact and strictly convex
hypersurface immersed in $\mathbb{R}^{n+1}$ with $n\geq2$. Then for
any continuous function $\kappa(t)$, there exists a unique, smooth
solution to evolution equation (\ref{1.1}) on a maximal time
interval $[0,T_{\mathrm{max}})$. If additionally the following limit
exists and satisfies
 \begin{eqnarray*}
 \lim\limits_{t\rightarrow{T_{\mathrm{max}}}}\kappa(t)=\overline{\kappa} \qquad
 \mathrm{and} \qquad |\overline{\kappa}|<+\infty,
 \end{eqnarray*}
 then we have

 (I) If $\Xi=\infty$, then $T_{\mathrm{max}}<\infty$ and the flow
 (\ref{1.1}) converges uniformly to a point as
 $t\rightarrow{T_{\mathrm{max}}}$. Moreover, the normalized
 equation (\ref{2.6}) has a solution
 $\widetilde{X}(x,\widetilde{t})$ for all times
 $0\leq\widetilde{t}\leq\infty$, and its hypersurfaces
 $\widetilde{M}(x,\widetilde{t})=\widetilde{M}_{\widetilde{t}}$ converge to a round sphere of area $|M_0|$ in the
 $C^{\infty}$-topology as $\widetilde{t}\rightarrow\infty$.

 (II) If $0<\Xi<\infty$, then $T_{\mathrm{max}}=\infty$ and the
 solutions to (\ref{1.1}) converge uniformly to a sphere in the
 $C^{\infty}$-topology as $t\rightarrow\infty$.

 (III) If $\Xi=0$, then $\overline{\kappa}\geq0$ and
 $T_{\mathrm{max}}=\infty$. Moreover, if $\overline{\kappa}>0$, the
 solutions to (\ref{1.1}) expand uniformly to $\infty$ as
 $t\rightarrow\infty$, and the limit of the rescaled solutions to (\ref{2.6}) must be a round
sphere of total area $|M_0|$ if they converge to a smooth
hypersurface.
 \end{theorem}

 \begin{remark}\label{remarkadd1}
 \rm{ Here we want to reveal the difference between the flow (\ref{1.1})
 and the MCF by an example, through which readers can find that the flow
 (\ref{1.1}) is not a simple and trivial extension of the classical
 MCF. Now, if the $n$-dimensional initial hypersurface $M_{0}$ is a
 sphere with radius $r_{0}$, clearly, it can be represented by
 \begin{eqnarray*}
X_{0}(r_{0},\theta_{1},\ldots,\theta_{n}):=(r_{0}\cos(\theta_{1}),r_{0}\sin(\theta_{1})\cos(\theta_{2}),r_{0}\sin(\theta_{1})\sin(\theta_{2})
\cos(\theta_{3}),\ldots,\nonumber\\
r_{0}\sin(\theta_{1})\ldots{\sin}(\theta_{n-1})\cos(\theta_{n}),r_{0}\sin(\theta_{1})\ldots{\sin}(\theta_{n-1})\sin(\theta_{n})),
\end{eqnarray*}
where $r_{0}>0$ and
$(\theta_{1},\ldots,\theta_{n-1},\theta_{n})\in{\mathbb{S}^{n}}$.
Then the flow (\ref{1.1}) becomes
\begin{eqnarray} \label{keyinequality1}
\left\{
\begin{array}{ll}
\frac{\partial}{\partial{t}}r(t)=-\frac{n}{r(t)}+\kappa(t)r(t), \\
r(0)=r_{0}, & \quad
\end{array}
\right.
\end{eqnarray}
since in this case the evolving hypersurfaces $M_{t}$
($0<t<T_{\mathrm{max}}$) should be spheres under the flow
(\ref{1.1}) and can be represented by
 \begin{eqnarray*}
X_{t}(r_{0},\theta_{1},\ldots,\theta_{n}):=(r(t)\cos(\theta_{1}),r_{t}\sin(\theta_{1})\cos(\theta_{2}),r(t)\sin(\theta_{1})\sin(\theta_{2})
\cos(\theta_{3}),\ldots,\nonumber\\
r(t)\sin(\theta_{1})\ldots{\sin}(\theta_{n-1})\cos(\theta_{n}),r(t)\sin(\theta_{1})\ldots{\sin}(\theta_{n-1})\sin(\theta_{n})).
\end{eqnarray*}
In fact, the assertion that $M_{t}$ ($0<t<T_{\mathrm{max}}$) is a
sphere can be obtained by the fact that the flow (\ref{1.1}) can
preserve the property of being totally umbilical, i.e.
$h_{ij}=Hg_{ij}/n$ (cf. Lemma \ref{lemma1}).  The first equation of
(\ref{keyinequality1}) is a Bernoullie equation, and by direct
computation, we can get
 \begin{eqnarray} \label{solution1}
r(t)=\left(r_{0}^{2}-2n\int_{0}^{t}e^{-2\int_{0}^{\tau}\kappa(\xi)d\xi}d\tau\right)^{1/2}\cdot{e^{\int_{0}^{t}\kappa(\tau)d\tau}}.
\end{eqnarray}
Clearly, from (\ref{solution1}) we know that the contraction or
expansion of $M_{t}$ depends on $\kappa(t)$ and $r_{0}$, and we can
also get information of $T_{\mathrm{max}}$ by considering the first
zero-point (if exists) of the function
$r_{0}^{2}-2n\int_{0}^{t}e^{-2\int_{0}^{\tau}\kappa(\xi)d\xi}d\tau$.
More precisely, if there exists some $t_{0}<+\infty$ such that
$r_{0}^{2}/2n=\int_{0}^{t_{0}}e^{-2\int_{0}^{\tau}\kappa(\xi)d\xi}d\tau$,
then we have $T_{\mathrm{max}}=t_{0}$, i.e. $M_{t}$ contracts to a
single point at $t_0$; if there does not exist, then
$T_{\mathrm{max}}=+\infty$, i.e. $M_{t}$ expands to infinity. In
order to let readers realize this clearly, we would like to
investigate several different $\kappa(t)$ which let the flow
(\ref{1.1}) have different behaviors. For instance, if we choose
$\kappa(t)=1/(t+1)$, then by (\ref{solution1}) we have
\begin{eqnarray*}
r(t)=\left(r_{0}^{2}-2n+\frac{2n}{t+1}\right)^{1/2}\cdot(t+1).
\end{eqnarray*}
Clearly, if $0<r_{0}<\sqrt{2n}$, then
$T_{\mathrm{max}}=r_{0}^{2}/(2n-r_{0}^{2})<\infty$, and $M_{t}$
contracts to a single point as $t\rightarrow T_{\mathrm{max}}$; if
$\sqrt{2n}\leq{r_{0}}<\infty$, then $T_{\mathrm{max}}=+\infty$, and
$M_{t}$ expands uniformly to $\infty$ as $t\rightarrow\infty$. If we
choose $\kappa(t)=-1/(t+1)$, then by (\ref{solution1}) we have
\begin{eqnarray*}
r(t)=\left[r_{0}^{2}-2n\frac{(t+1)^3}{3}+\frac{2n}{3}\right]^{1/2}\cdot\frac{1}{t+1}.
\end{eqnarray*}
Clearly, no matter how much $r_{0}$ is, $M_{t}$ contracts to a
single point as $t\rightarrow T_{\mathrm{max}}$ and
$T_{\mathrm{max}}=\sqrt[3]{1+\frac{3r_{0}^{2}}{2n}}-1<+\infty$. From
these two examples, we know that different $\kappa(t)$ might let the
flow (\ref{1.1}) have different behaviors (i.e. contraction and
expansion are all possible). However, Huisken \cite{Hus} proved that
an $n$-dimensional smooth, compact and strictly convex hypersurface
immersed in $\mathbb{R}^{n+1}$ with $n\geq2$ evolves under the MCF
\emph{would only contract to single point at a finite time}. In
fact, if $M_{0}$ is a sphere which can be represented as above, then
the MCF should become
\begin{eqnarray*}
\left\{
\begin{array}{ll}
\frac{\partial}{\partial{t}}r(t)=-\frac{n}{r(t)}, \\
r(0)=r_{0}. & \quad
\end{array}
\right.
\end{eqnarray*}
So, $r(t)=\sqrt{r_{0}^{2}-2nt}$ and the maximal time is
$T_{\mathrm{max}}=\frac{r_{0}^{2}}{2n}$. Clearly, even in this
special setting (i.e. the initial hypersurface is a sphere), the
situation of our flow (\ref{1.1}) is more complicated than that of
the MCF. Hence, the flow (\ref{1.1}) cannot be seen as a simple
extension of the MCF. From the above argument, one can realize that
one needs to study the function $\kappa(t)$ and might also (if
needed) the diameter (or equivalently, the mean curvature) of the
initial hypersurface if he or she wants to investigate behaviors of
the evolving hypersurfaces under the flow (\ref{1.1}), and this
difficulty has been solved in \cite{lmw} by successfully finding a
breakthrough, i.e. discussing the limit $\Xi$ determined by
(\ref{2.8}), which in essence has relation with $\kappa(t)$ and the
mean curvature of the initial hypersurface. However, in the case of
the classical MCF, this problem does not exist. One cannot get
Theorem \ref{theoremmain} \emph{only} by applying Huisken's method
(i.e. $L^{p}$-estimate) in \cite{Hus}. In fact, to prove Theorem
\ref{theoremmain}, except the $L^{p}$-estimate tool, one might also
have to use other tools introduced in \cite{ban,lis} (see \cite{lmw}
for the details).

However, the above process might only works for this special case
(i.e. the initial hypersurface is a sphere) in which we can compute
$X_{t}$ directly. Actually, even in this special case when
$\kappa(t)$ is complicated, for instance, choose
$\kappa(t)=\sqrt{1+\frac{1}{t+4}\sqrt{\frac{1}{t+3}\sqrt{\frac{1}{t+2}\sqrt{\frac{1}{t+1}}}}}$,
then it is not easy to compute directly. Of course, in this case, we
might get the numerical value of $T_{\mathrm{max}}=t_{0}<\infty$ (if
exits) by software once $r_0$ and $n$ are given. Therefore, \emph{it
should be interesting to know how $M_{t}$ behaves and
$T_{\mathrm{max}}$ once $\kappa(t)$ is given and the initial
hypersurface $M_{0}$ is not so special as above}. Theorem
\ref{theoremmain} can supply us this possibility. In fact, if
$\kappa(t)$ is given, then the rescaled factor $\phi(t)$ might be
solved by (\ref{2.5}) (if feasible), and then applying Theorem
\ref{theoremmain} the behavior of $M_{t}$ and the information of
$T_{\mathrm{max}}$ can be known. }
 \end{remark}

\section{Evolution equations for the first eigenvalues of the Laplace and the $p$-Laplace operators}
\renewcommand{\thesection}{\arabic{section}}
\renewcommand{\theequation}{\thesection.\arabic{equation}}
\setcounter{equation}{0} \setcounter{maintheorem}{0}

In this section, based on the evolution equations mentioned in
Section 2, we would like to derive evolution equations for the first
eigenvalues of the Laplacian and the $p$-Laplacian as follows.

\begin{theorem} \label{theorem1}
Let $\lambda_{1}(t)$ be the first non-zero closed eigenvalue of the
Laplacian on an $n$-dimensional compact and strictly convex
hypersurface $M_{t}$ which evolves by the forced MCF (\ref{1.1}),
and let $u$ be the normalized eigenfunction corresponding to
$\lambda_{1}$, i.e. $-\Delta{u}=\lambda_{1}u$ and
$\int_{M_t}u^{2}=1$. Then we have
\begin{eqnarray} \label{1.2}
\frac{d}{dt}\lambda_{1}(t)=-2\lambda_{1}\kappa(t)+2\int_{M_t}Hh^{ij}\nabla_{i}u\nabla_{j}u+2\int_{M_t}uH\nabla_{i}h^{ij}\nabla_{j}u.
\end{eqnarray}
Similarly, under the normalized flow (\ref{2.6}), we have
\begin{eqnarray*}
\frac{d}{d\tilde{t}}\widetilde{\lambda}_{1}(\tilde{t})=-\frac{2\widetilde{h}}{n}\cdot\widetilde{\lambda}_{1}(\tilde{t})+2\int_{\widetilde{M}_{\tilde{t}}}\widetilde{H}\cdot\widetilde{h}^{ij}\nabla_{i}u\nabla_{j}u
+2\int_{\widetilde{M}_{\tilde{t}}}u\widetilde{H}\nabla_{i}\widetilde{h}^{ij}\nabla_{j}u,
\end{eqnarray*}
where $\widetilde{\lambda}_{1}(\tilde{t})$ is the first non-zero
closed eigenvalue of the Laplacian on the rescaled hypersurface
$\widetilde{M}_{\tilde{t}}$.
\end{theorem}

\begin{proof}
Let $u$ be an eigenfunction of the first non-zero closed
eigenvalue $\lambda_1$ of $\Delta$ on the evolving compact
hypersurface $M_t$. For simplicity, we normalize the function $u$,
i.e. $\int_{M_t}u^{2}=1$. By (\ref{2.1}), we know that $u$ also
satisfies
\begin{eqnarray*}
-\Delta{u}=\lambda_1{u},  \qquad \mathrm{where} \quad \int_{M_t}u=0.
\end{eqnarray*}
Clearly, we have
\begin{eqnarray} \label{3.1}
-\frac{\partial}{\partial{t}}\left(\Delta{u}\right)=\left(\frac{d}{dt}\lambda_{1}\right)u+\lambda_{1}\frac{\partial{u}}{\partial{t}}
\end{eqnarray}
by taking derivatives with respect to $t$ for the above equation. By
multiplying $u$ to both sides of (\ref{3.1}) and then integrating
over $M_t$, we have
\begin{eqnarray*}
-\int_{M_t}u\frac{\partial}{\partial{t}}\left(\Delta{u}\right)=
\left(\frac{d}{dt}\lambda_{1}\right)\int_{M_t}u^{2}+\lambda_{1}\int_{M_t}u\frac{\partial{u}}{\partial{t}}.
\end{eqnarray*}
Therefore, we can obtain
\begin{eqnarray} \label{3.2}
\frac{d}{dt}\lambda_{1}=-\int_{M_t}u\frac{\partial}{\partial{t}}\left(\Delta{u}\right)-\lambda_{1}\int_{M_t}u\frac{\partial{u}}{\partial{t}}.
\end{eqnarray}
Hence, if we want to get the evolution equation of $\lambda_{1}$, we
need to derive the evolution equation of $\Delta{u}$ under the flow
(\ref{1.1}). First, by (\ref{2.3}) we have
\begin{eqnarray*}
\frac{\partial}{\partial{t}}g^{ij}=-g^{im}\left(\frac{\partial}{\partial{t}}g_{mq}\right)g^{qj}
=2g^{im}\left[Hh_{mq}-\kappa(t)g_{mq}\right]g^{qj}=2Hg^{im}h_{mq}g^{qj}-2\kappa(t)g^{ij},
\end{eqnarray*}
which implies
\begin{eqnarray}  \label{3.3}
\frac{\partial}{\partial{t}}(\Delta{u})&=&\frac{\partial}{\partial{t}}\left(g^{ij}\nabla_{i}\nabla_{j}u\right)\nonumber\\
&=&\frac{\partial}{\partial{t}}(g^{ij})\nabla_{i}\nabla_{j}u+g^{ij}\frac{\partial}{\partial{t}}(\nabla_{i}\nabla_{j}u)\nonumber\\
&=&2\left[Hg^{im}h_{mq}g^{qj}-\kappa(t)g^{ij}\right]\nabla_{i}\nabla_{j}u+g^{ij}\frac{\partial}{\partial{t}}
\left(\frac{\partial^{2}u}{\partial{x_i}\partial{x_j}}-\Gamma^{m}_{ij}\frac{\partial{u}}{\partial{x_m}}\right)\nonumber\\
&=&2Hg^{im}h_{mq}g^{qj}\nabla_{i}\nabla_{j}u-2\kappa(t)\Delta{u}+\Delta\frac{\partial{u}}{\partial{t}}-g^{ij}
\frac{\partial}{\partial{t}}\left(\Gamma^{m}_{ij}\right)\frac{\partial{u}}{\partial{x_m}}.
\end{eqnarray}
On the other hand, we have
\begin{eqnarray*}
g^{ij}\frac{\partial}{\partial{t}}\left(\Gamma^{m}_{ij}\right)&=&\frac{1}{2}g^{ij}g^{ml}
\left(\nabla_{i}\frac{\partial{g_{jl}}}{\partial{t}}+\nabla_{j}\frac{\partial{g_{il}}}{\partial{t}}-
\nabla_{l}\frac{\partial{g_{ij}}}{\partial{t}}\right)\\
&=&\frac{1}{2}g^{ij}g^{ml}
\Big{\{}\nabla_{i}\left[-2Hh_{jl}+2\kappa(t)g_{jl}\right]+\nabla_{j}\left[-2Hh_{il}+2\kappa(t)g_{il}\right]-\\
&& \qquad
\nabla_{l}\left[-2Hh_{ij}+2\kappa(t)g_{ij}\right]\Big{\}}\\
&=&-2\nabla_{i}H\cdot{g^{ij}}g^{ml}h_{jl}.
\end{eqnarray*}
Substituting the above equality into (\ref{3.3}) results in
 \begin{eqnarray} \label{3.4}
\frac{\partial}{\partial{t}}(\Delta{u})=2Hg^{im}h_{mq}g^{qj}\nabla_{i}\nabla_{j}u-2\kappa(t)\Delta{u}+\Delta\frac{\partial{u}}{\partial{t}}+2\nabla_{i}H\cdot{g^{ij}}g^{ml}h_{jl}\frac{\partial{u}}{\partial{x_m}}.
 \end{eqnarray}
 By substituting (\ref{3.4}) into (\ref{3.2}), and then integrating by parts, we have
 \begin{eqnarray} \label{3.5}
 \frac{d}{dt}\lambda_{1}&=&-\int_{M_t}u\left[2Hg^{im}h_{mq}g^{qj}\nabla_{i}\nabla_{j}u-2\kappa(t)\Delta{u}+\Delta\frac{\partial{u}}{\partial{t}}+2\nabla_{i}H\cdot{g^{ij}}g^{ml}h_{jl}\nabla_{m}u\right]-\lambda_{1}\int_{M_t}u\frac{\partial{u}}{\partial{t}}\nonumber\\
 &=&2\int_{M_t}Hh^{ij}\nabla_{i}u\nabla_{j}u-2\lambda_{1}\kappa(t)-\int_{M_t}u\left(\Delta\frac{\partial{u}}{\partial{t}}\right)-\lambda_{1}\int_{M_t}u\frac{\partial{u}}{\partial{t}}\nonumber\\
 &=&-2\lambda_{1}\kappa(t)+2\int_{M_t}Hh^{ij}\nabla_{i}u\nabla_{j}u+2\int_{M_t}uH\nabla_{i}h^{ij}\nabla_{j}u,
\end{eqnarray}
where $h^{ij}=g^{im}h_{mq}g^{qj}$. Here the last equality in
(\ref{3.5}) holds since
\begin{eqnarray*}
\int_{M_t}u\left(\Delta\frac{\partial{u}}{\partial{t}}\right)=\int_{M_t}\Delta{u}\frac{\partial{u}}{\partial{t}}
=-\lambda_{1}\int_{M_t}u\frac{\partial{u}}{\partial{t}}.
\end{eqnarray*}
This completes the proof of (\ref{1.2}).

Similarly, under the normalized flow (\ref{2.6}), we can obtain
\begin{eqnarray*}
\frac{d}{d\tilde{t}}\lambda_{1}(\tilde{t})=-\frac{2\widetilde{h}}{n}\cdot\lambda_{1}(\tilde{t})+2\int_{\widetilde{M}_{\tilde{t}}}\widetilde{H}\cdot\widetilde{h}^{ij}\nabla_{i}u\nabla_{j}u+2\int_{\widetilde{M}_{\tilde{t}}}u\widetilde{H}\nabla_{i}\widetilde{h}^{ij}\nabla_{j}u,
\end{eqnarray*}
since the evolution equations (\ref{2.3}) and (\ref{2.7})
\emph{almost}
 have the same form except the function $\kappa(t)$ replaced by
 $\widetilde{h}/n$ with
 $\widetilde{h}=\phi^{-2}h=\int_{\widetilde{M_{\widetilde{t}}}}\widetilde{H}^{2}/\int_{\widetilde{M_{\widetilde{t}}}}$.
 \end{proof}

\begin{remark}
\rm{ Here we want to emphasize one thing, that is, we need to
require that $M_t$ should be compact on a prescribed time interval,
since the compactness of $M_t$ can assure the existence of the
eigenvalues of the Laplace and the $p$-Laplace operators. This
implies that it cannot be avoided investigating the evolving
behavior of the forced flow (\ref{1.1}). In fact, by Theorem
\ref{theoremmain}, we know that it is feasible to consider the
evolution equation (\ref{1.2}) of the first nonzero closed
eigenvalue of the Laplace operator on $[0,T_{\mathrm{m}})$ with
$T_{\mathrm{m}}$ defined by
\begin{eqnarray}  \label{key}
T_{\mathrm{m}}=\left\{
\begin{array}{lll}
T_{\mathrm{max}}, \qquad &\mathrm{if}~~0<\Xi\leq\infty, \\
T<T_{\mathrm{max}}, \qquad &\mathrm{if}~~\Xi=0,
\end{array}
\right.
\end{eqnarray}
where $\Xi$ is the limit given by (\ref{2.8}) and
$[0,T_{\mathrm{max}})$ corresponds to the maximal time interval of
the flow (\ref{1.1}). Clearly, on $[0,T_{\mathrm{m}})$, the evolving
hypersurface $M_{t}$ is compact. }
\end{remark}

In the case of the $p$-Laplace operator, since we do not know
whether the first nonzero closed eigenvalue $\lambda_{1,p}(t)$ of
$\Delta_{p}$ is differentiable under the forced flow (\ref{1.1}) or
not, it seems like that we cannot use a similar method to that of
the proof of Theorem \ref{theorem1}. However, in fact, we can use a
similar method to the one in \cite{caox,caox2} to avoid discussing
the differentiation of $\lambda_{1,p}(t)$ under the flow
(\ref{1.1}). More precisely, on the time interval
$[0,T_{\mathrm{m}})$ where the flow (\ref{1.1}) exists and $M_t$ is
compact, we can define a smooth function $\lambda_{1,p}(u,t)$ as
follows
\begin{eqnarray} \label{3.6}
\lambda_{1,p}(u,t):=-\int_{M_t}\Delta_{p}u(x,t)\cdot{u(x,t)}dv_{t}=\int_{M_t}|\nabla{u}|^{p}dv_{t},
\end{eqnarray}
where $u(x,t)$ is an arbitrary smooth function satisfying
\begin{eqnarray}  \label{3.7}
\int_{M_t}|u(x,t)|^{p}=1  \qquad \mathrm{and} \quad
\int_{M_t}|u(x,t)|^{p-2}u(x,t)=0.
\end{eqnarray}
Clearly, for any $t\in[0,T_{\mathrm{m}})$, if, furthermore, $u(x,t)$
is the eigenfunction of the first eigenvalue $\lambda_{1,p}(t)$,
then, by (\ref{3.7}), we have
\begin{eqnarray*}
\lambda_{1,p}(u,t)=-\int_{M_t}\Delta_{p}u(x,t)\cdot{u(x,t)}=\lambda_{1,p}(t)\int_{M_t}|u(x,t)|^{p}=\lambda_{1,p}(t).
\end{eqnarray*}

Now, by using the function $\lambda_{1,p}(u,t)$ defined by
(\ref{3.6}), we can prove the following result.

\begin{theorem} \label{theorem2}
Let $\lambda_{1,p}(t)$ be the first non-zero closed eigenvalue of
the $p$-Laplacian ($1<p<\infty$) on an $n$-dimensional compact and
strictly convex hypersurface $M_{t}$ which evolves by the forced MCF
(\ref{1.1}), and let $u$ be the eigenfunction of $\lambda_{1,p}(t)$
at time $t\in[0,T_{\mathrm{m}})$ satisfying $\int_{M_t}u^{p}=1$,
where $T_{\mathrm{m}}$ is defined by (\ref{key}). Let
$\lambda_{1,p}(u,t)$ be the smooth function defined by (\ref{3.6}).
Then at time $t$ we have
\begin{eqnarray} \label{1.3}
\frac{d}{dt}\lambda_{1,p}(u,t)=-p\kappa(t)\lambda_{1,p}(t)+p\int_{M_t}|\nabla{u}|^{p-2}Hh^{ij}\nabla_{i}u\cdot\nabla_{j}u+
2\int_{M_t}|\nabla{u}|^{p-2}uH\nabla_{i}h^{ij}\nabla_{j}u.
\end{eqnarray}
Similarly, under the normalized flow (\ref{2.6}), we have
\begin{eqnarray*}
\frac{d}{d\tilde{t}}\widetilde{\lambda}_{1,p}(u,\tilde{t})=-\frac{p\widetilde{h}}{n}\cdot\widetilde{\lambda}_{1,p}(\tilde{t})+p\int_{\widetilde{M}_{\tilde{t}}}|\nabla{u}|^{p-2}\widetilde{H}\cdot\widetilde{h}^{ij}\nabla_{i}u\cdot\nabla_{j}u+
2\int_{\widetilde{M}_{\tilde{t}}}|\nabla{u}|^{p-2}u\widetilde{H}\nabla_{i}\widetilde{h}^{ij}\nabla_{j}u
\end{eqnarray*}
at time $\tilde{t}\in[0,\widetilde{T}_{\mathrm{m}})$. Here
$\widetilde{T}_{\mathrm{m}}:=\int_{0}^{T_{\mathrm{m}}}\phi^{2}(s)ds$
with $\phi(t)$ the rescaled factor determined by (\ref{2.5}).
Moreover, $\widetilde{\lambda}_{1,p}(\tilde{t})$ is the first
nonzero closed eigenvalue of the $p$-Laplacian on the rescaled
hypersurface $\widetilde{M}_{\tilde{t}}$, and
$\widetilde{\lambda}_{1,p}(u,\tilde{t})$ is a smooth function
defined by
\begin{eqnarray*}
\widetilde{\lambda}_{1,p}(u,\tilde{t}):=-\int_{\widetilde{M}_{\tilde{t}}}\Delta_{p}u(x,\tilde{t})\cdot{u(x,\tilde{t})},
\end{eqnarray*}
where $u(x,\tilde{t})$ is an arbitrary smooth function satisfying
\begin{eqnarray*}
\int_{\widetilde{M}_{\tilde{t}}}|u(x,\tilde{t})|^{p}=1  \qquad
\mathrm{and} \quad
\int_{\widetilde{M}_{\tilde{t}}}|u(x,\tilde{t})|^{p-2}u(x,\tilde{t})=0.
\end{eqnarray*}
\end{theorem}

\begin{proof}
  Taking derivatives with respect to $t$ on both sides of
(\ref{3.6}), we have
\begin{eqnarray}  \label{3.8}
-\frac{d}{dt}\lambda_{1,p}(u,t)=\frac{d}{dt}\int_{M_t}u\Delta_{p}udv_{t}.
\end{eqnarray}
For convenience in the computation below, set $B=|\nabla{u}|^{p-2}$,
and then $\Delta_{p}u=\mathrm{div}[B(\nabla{u})]$. Furthermore, we
have
\begin{eqnarray*}
\frac{\partial}{\partial{t}}\int_{M_t}u\Delta_{p}udv_{t}&=&\frac{\partial}{\partial{t}}\int_{M_t}g^{ij}\nabla_{i}[B(\nabla_{j}u)]udv_{t}\\
&=&\int_{M_t}\frac{\partial}{\partial{t}}\left[g^{ij}\nabla_{i}B\nabla_{j}u+B\Delta{u}\right]udv_{t}+\int_{M_t}g^{ij}\nabla_{i}[B(\nabla_{j}u)](u_{t}dv_{t}+u(dv_{t})_{t})\\
&=&\int_{M_t}\left[\left(\frac{\partial}{\partial{t}}g^{ij}\right)\nabla_{i}B\nabla_{j}u+g^{ij}\nabla_{i}B_{t}\nabla_{j}u+g^{ij}\nabla_{i}B\nabla_{j}u_{t}+B_{t}\Delta{u}+
B\frac{\partial}{\partial{t}}(\Delta{u})\right]udv_{t}\\
&&\qquad+\int_{M_t}g^{ij}\nabla_{i}[B(\nabla_{j}u)](u_{t}dv_{t}+u(dv_{t})_{t}),
\end{eqnarray*}
where, except $dv_{t}$, the subscript $(\cdot)_{t}$ means taking
derivative with respect to $t$ for the prescribed function.
Substituting the corresponding evolution equations of $g^{ij}$,
$\Delta{u}$ under the flow (\ref{1.1}) derived in the proof of
Theorem \ref{theorem1} into the above equality results in
\begin{eqnarray} \label{3.9}
&&\frac{\partial}{\partial{t}}\int_{M_t}u\Delta_{p}udv_{t}=\int_{M_t}u\Big{[}\left(2Hh^{ij}-2\kappa(t)g^{ij}\right)\nabla_{i}B\nabla_{j}u+g^{ij}\nabla_{i}B_{t}\nabla_{j}u+g^{ij}\nabla_{i}B\nabla_{j}u_{t}+B_{t}\Delta{u}+\nonumber\\
&&\qquad B\left(2Hh^{ij}\nabla_{i}\nabla_{j}u-2\kappa(t)\Delta{u}+\Delta\frac{\partial{u}}{\partial{t}}+2\nabla_{i}H\cdot{h^{im}}\nabla_{m}u\right)\Big{]}dv_{t}+\nonumber\\
&&\qquad\qquad\int_{M_t}g^{ij}\nabla_{i}[B(\nabla_{j}u)](u_{t}dv_{t}+u(dv_{t})_{t})\nonumber\\
 &&=\int_{M_t}u\left(2Hh^{ij}-2\kappa(t)g^{ij}\right)\nabla_{i}(B\nabla_{j}u)dv_{t}+\int_{M_t}g^{ij}\nabla_{i}(B_{t}\nabla_{j}u)udv_{t}+
 \int_{M_t}g^{ij}\nabla_{i}(B\nabla_{j}u_{t})udv_{t}+\nonumber\\
&&\qquad
+2\int_{M_t}Bu\nabla_{i}H\cdot{h^{im}}\nabla_{m}udv_{t}+\int_{M_t}g^{ij}\nabla_{i}[B(\nabla_{j}u)](u_{t}dv_{t}+u(dv_{t})_{t})\nonumber\\
&&=\int_{M_t}u\left(2Hh^{ij}-2\kappa(t)g^{ij}\right)\nabla_{i}(B\nabla_{j}u)dv_{t}-\int_{M_t}g^{ij}B_{t}\nabla_{i}u\cdot\nabla_{j}udv_{t}-\int_{M_t}g^{ij}B\nabla_{i}u\cdot\nabla_{j}u_{t}dv_{t}\nonumber\\
&&\qquad\qquad
+2\int_{M_t}Bu\nabla_{i}H\cdot{h^{im}}\nabla_{m}udv_{t}+\int_{M_t}g^{ij}\nabla_{i}[B(\nabla_{j}u)](u_{t}dv_{t}+u(dv_{t})_{t}).
\end{eqnarray}
Since
\begin{eqnarray*}
B_{t}&=&\frac{\partial{B}}{\partial{t}}=\frac{\partial}{\partial{t}}|\nabla{u}|^{p-2}\\
&=&\frac{\partial}{\partial{t}}\left(|\nabla{u}|^{2}\right)^{\frac{p-2}{2}}\\
&=&\frac{\partial}{\partial{t}}\left(g^{ij}\nabla_{i}u\nabla_{j}u\right)^{\frac{p-2}{2}}\\
&=&(p-2)|\nabla{u}|^{p-4}\left[Hh^{ij}-\kappa(t)g^{ij}\right]\nabla_{i}u\nabla_{j}u+(p-2)|\nabla{u}|^{p-4}g^{ij}\nabla_{i}u_{t}\cdot\nabla_{j}u,
\end{eqnarray*}
then substituting the above equality into (\ref{3.9}) yields
\begin{eqnarray} \label{3.10}
\frac{\partial}{\partial{t}}\int_{M_t}u\Delta_{p}udv_{t}&=&2\int_{M_t}u\left[Hh^{ij}-\kappa(t)g^{ij}\right]\nabla_{i}(B\nabla_{j}u)dv_{t}
-(p-2)\int_{M_t}|\nabla{u}|^{p-2}\cdot\nonumber\\
&&\left[Hh^{ij}-\kappa(t)g^{ij}\right]\nabla_{i}u\nabla_{j}udv_{t}-(p-1)\int_{M_t}|\nabla{u}|^{p-2}g^{ij}\nabla_{i}u_{t}\cdot\nabla_{j}{u}dv_{t}+\nonumber\\
&&2\int_{M_t}Bu\nabla_{i}H\cdot{h^{im}}\nabla_{m}udv_{t}+\int_{M_t}g^{ij}\nabla_{i}[B(\nabla_{j}u)](u_{t}dv_{t}+u(dv_{t})_{t})\nonumber\\
&=&-p\int_{M_t}B\left[Hh^{ij}-\kappa(t)g^{ij}\right]\nabla_{i}u\nabla_{j}udv_{t}-2\int_{M_t}BuH\nabla_{i}h^{ij}\nabla_{j}udv_{t}-\nonumber\\
\lefteqn{(p-1)\int_{M_t}Bg^{ij}\nabla_{i}u_{t}\nabla_{j}{u}dv_{t}+\int_{M_t}g^{ij}\nabla_{i}[B(\nabla_{j}u)](u_{t}dv_{t}+u(dv_{t})_{t}).}
\end{eqnarray}
By divergence theorem, we have
\begin{eqnarray*}
-(p-1)\int_{M_t}Bg^{ij}\nabla_{i}u_{t}\nabla_{j}{u}dv_{t}=(p-1)\int_{M_t}g^{ij}\nabla_{i}[B(\nabla_{j}u)]u_{t}dv_{t}.
\end{eqnarray*}
Substituting the above equality into (\ref{3.10}), we have
\begin{eqnarray}  \label{3.11}
\frac{\partial}{\partial{t}}\int_{M_t}u\Delta_{p}udv_{t}=-p\int_{M_t}B\left[Hh^{ij}-\kappa(t)g^{ij}\right]\nabla_{i}u\nabla_{j}udv_{t}-2\int_{M_t}BuH\nabla_{i}h^{ij}\nabla_{j}udv_{t}+\nonumber\\
\int_{M_t}g^{ij}\nabla_{i}[B(\nabla_{j}u)](pu_{t}dv_{t}+u(dv_{t})_{t}).
\qquad
\end{eqnarray}
If now $u$ is the eigenfunction of the first non-zero closed
eigenfunction $\lambda_{1,p}(t)$, then, as pointed out before, we
have $\lambda_{1,p}(u,t)=\lambda_{1,p}(t)$ and
$\Delta_{p}u=-\lambda_{1,p}(t)|u|^{p-2}u$. By applying this fact and
(\ref{3.7}), we can obtain
\begin{eqnarray*}
\frac{d}{dt}\int_{M_t}|u(x,t)|^{p}dv_{t}&=&\frac{d}{dt}\int_{M_t}B\cdot\left(g^{ij}\nabla_{i}u\nabla_{j}u\right)dv_{t}\\
&=&\int_{M_t}Bu(pu_{t}dv_{t}+u(dv_{t})_{t})\\
&=&-(\lambda_{1,p})^{-1}\int_{M_t}g^{ij}\nabla_{i}[B(\nabla_{j}u)](pu_{t}dv_{t}+u(dv_{t})_{t})=0.
\end{eqnarray*}
Together the above equality with (\ref{3.11}), we have
\begin{eqnarray} \label{3.12}
\frac{\partial}{\partial{t}}\int_{M_t}u\Delta_{p}udv_{t}=-p\int_{M_t}B\left[Hh^{ij}-\kappa(t)g^{ij}\right]\nabla_{i}u\nabla_{j}udv_{t}-2\int_{M_t}BuH\nabla_{i}h^{ij}\nabla_{j}udv_{t}.
\end{eqnarray}
By substituting (\ref{3.6}) and (\ref{3.8}) into (\ref{3.12}), we
have
\begin{eqnarray*}
\frac{d}{dt}\lambda_{1,p}(u,t)&=&p\int_{M_t}B\left[Hh^{ij}-\kappa(t)g^{ij}\right]\nabla_{i}u\nabla_{j}udv_{t}+2\int_{M_t}BuH\nabla_{i}h^{ij}\nabla_{j}udv_{t}\nonumber\\
&=&p\int_{M_t}BHh^{ij}\nabla_{i}u\cdot\nabla_{j}udv_{t}+2\int_{M_t}BuH\nabla_{i}h^{ij}\nabla_{j}udv_{t}-p\kappa(t)\lambda_{1,p}(t),
\end{eqnarray*}
which completes the proof of (\ref{1.3}).

Similarly, under the normalized flow (\ref{2.5}), we can obtain
\begin{eqnarray*}
\frac{d}{d\tilde{t}}\lambda_{1,p}(u,\tilde{t})=p\int_{\widetilde{M}_{\tilde{t}}}|\nabla{u}|^{p-2}\widetilde{H}\cdot\widetilde{h}^{ij}\nabla_{i}u\cdot\nabla_{j}udv_{\tilde{t}}+
2\int_{\widetilde{M}_{\tilde{t}}}|\nabla{u}|^{p-2}u\widetilde{H}\nabla_{i}\widetilde{h}^{ij}\nabla_{j}udv_{\tilde{t}}-\frac{p\widetilde{h}}{n}\cdot\lambda_{1,p}(\tilde{t}),
\end{eqnarray*}
which completes the second claim of Theorem \ref{theorem2}.
\end{proof}

\begin{remark}\rm{Since (\ref{1.3}) does not depend on the particular evolution
of $u$, we have $d\lambda_{1,p}(u,t)/dt=d\lambda_{1,p}(t)/dt$ at
some time $t$. Clearly, at some time $t\in[0,T_{\mathrm{m}})$,
(\ref{1.2}) can be directly obtained by choosing $p=2$ in
(\ref{1.3}), which gives an explanation to the fact that the
nonlinear Laplacian $\Delta_{p}$ is an extension of the linear
Laplacian $\Delta$ from the viewpoint of the evolution equation.
\emph{Because of this, one may ask that maybe it is not necessary to
derive (\ref{1.2}) independently}. However, readers can find that
the way for proving (\ref{1.2}) cannot be used to derive (\ref{1.3})
directly because of indeterminacy of the differentiability of
$\lambda_{1,p}(t)$, and we have to construct a smooth function
$\lambda_{1,p}(u,t)$ defined by (\ref{3.6}) to overcome this
problem. This is the reason why we separately give evolution
equations of the first eigenvalues of the Laplace and the
$p$-Laplace operators.
 }
\end{remark}

\section{Lower bounds of the first eigenvalue of the Laplacian}
\renewcommand{\thesection}{\arabic{section}}
\renewcommand{\theequation}{\thesection.\arabic{equation}}
\setcounter{equation}{0} \setcounter{maintheorem}{0}

In this section, we would like to give lower bounds for the first
nonzero closed eigenvalue of the Laplace operator if additionally
the initial hypersurface $M_{0}$ satisfies the pinching condition
(\ref{pinchingc}). However, first, we want to show that this
pinching condition (\ref{pinchingc}) is preserved under the forced
MCF (\ref{1.1}), i.e. the evolving hypersurface $M_{t}$ also
satisfies (\ref{pinchingc}) for any $t\in[0,T_{\mathrm{max}})$. To
prove this, we need to use Hamilton's maximum principle for tensors
on manifolds (cf. \cite[Theorem 9.1]{rsh}). For convenience, we
prefer to list its details here.

\begin{theorem} \label{theoremhamilton}  (Hamilton)
Suppose that on $0\leq{t}<T$ the evolution equation
\begin{eqnarray*}
\frac{\partial}{\partial{t}}M_{ij}=\Delta{M_{ij}}+u^{k}\nabla_{k}M_{ij}+N_{ij}
\end{eqnarray*}
holds, where $N_{ij}=p(M_{ij},g_{ij})$, a polynomial in $M_{ij}$
formed by contracting products of $M_{ij}$ with itself using the
metric, satisfies the null-eigenvector condition below. If
$M_{ij}\geq0$ at $t=0$, then it remains so on $0\leq{t}<T$.
\end{theorem}

\begin{remark}\rm{
Here we would like to make an explanation to the so-called
\emph{null-eigenvector condition}. In fact,
$N_{ij}=p(M_{ij},g_{ij})$ satisfies the null-eigenvector condition
implies that for any null-eigenvector $X$ of $M_{ij}$, we have
$N_{ij}X^{i}X^{j}\geq0$. }
\end{remark}

By applying Theorem \ref{theoremhamilton}, we can prove the
following result.

\begin{lemma}  \label{lemma1}
If, in addition, the initial hypersurface $M_{0}$ satisfies the
pinching condition (\ref{pinchingc}), then the evolving hypersurface
$M_{t}$ remains so under the flow (\ref{1.1}) for any
$0\leq{t}<T_{\mathrm{max}}$.
\end{lemma}
\begin{proof}
By (\ref{pinchingc}), we have
\begin{eqnarray*}
h_{ij}=\alpha_{i}Hg_{ij}, \qquad \mathrm{on}~M_{0},
\end{eqnarray*}
that is,
\begin{eqnarray*}
\alpha_{i}Hg_{ij}\leq{h_{ij}}\leq\alpha_{i}Hg_{ij}, \qquad
\mathrm{on}~M_{0}.
\end{eqnarray*}
On the other hand, by (\ref{2.3}) and (\ref{2.4}), we have
\begin{eqnarray*}
&&\frac{\partial}{\partial{t}}\left(h_{ij}-\alpha_{i}Hg_{ij}\right)=\Delta\left(h_{ij}-\alpha_{i}Hg_{ij}\right)+|A|^{2}\left(h_{ij}-\alpha_{i}Hg_{ij}\right)
+\kappa(t)\left(h_{ij}-\alpha_{i}Hg_{ij}\right)-\nonumber\\
 &&\qquad\qquad\qquad\qquad\qquad  2H\left(h_{il}g^{lm}h_{mj}-\alpha_{i}Hh_{ij}\right).
\end{eqnarray*}
Now, we use Theorem \ref{theoremhamilton} to prove Lemma
\ref{lemma1}. In fact, we can choose
\begin{eqnarray*}
M_{ij}=h_{ij}-\alpha_{i}Hg_{ij}
\end{eqnarray*}
and
\begin{eqnarray*}
N_{ij}=|A|^{2}\left(h_{ij}-\alpha_{i}Hg_{ij}\right)
+\kappa(t)\left(h_{ij}-\alpha_{i}Hg_{ij}\right)-2H\left(h_{il}g^{lm}h_{mj}-\alpha_{i}Hh_{ij}\right).
\end{eqnarray*}
Clearly, $M_{ij}\geq0$ at $t=0$. It only needs to check that
$N_{ij}$ is nonnegative on the null-eigenvectors of $M_{ij}$. Assume
that, for some vector $X=\{X^{i}\}$, we have
\begin{eqnarray*}
h_{ij}X^{j}=\alpha_{i}HX_{i}.
\end{eqnarray*}
So, we can obtain
\begin{eqnarray*}
N_{ij}X^{i}X^{j}&=&\left[|A|^{2}+\kappa(t)\right]\left(\alpha_{i}HX_{i}X^{i}-\alpha_{i}Hg_{ij}X^{i}X^{j}\right)
-2H\left(h_{il}g^{lm}\alpha_{m}HX_{m}X^{i}-\alpha_{i}^{2}H^{2}X_{i}X^{i}\right)\nonumber\\
&=&\left[|A|^{2}+\kappa(t)\right]\left(\alpha_{i}HX_{i}X^{i}-\alpha_{i}Hg_{ij}X^{i}X^{j}\right)-2H\left(\alpha_{m}\alpha_{l}H^{2}g^{lm}X_{m}X_{l}-\alpha_{i}^{2}H^{2}X_{i}X^{i}\right)=0.
\end{eqnarray*}
Hence, $M_{ij}\geq0$ on $M_{t}$ for any $0\leq{t}<T_{\mathrm{max}}$,
i.e. $h_{ij}\geq\alpha_{i}Hg_{ij}$ for any
$t\in[0,T_{\mathrm{max}})$. Similarly, one can easily get
$h_{ij}\leq\alpha_{i}Hg_{ij}$ for any $0\leq{t}<T_{\mathrm{max}}$.
So, we have
\begin{eqnarray*}
h_{ij}=\alpha_{i}Hg_{ij}, \qquad
\mathrm{on}~M_{t}~~\mathrm{for}~~0\leq{t}<T_{\mathrm{max}},
\end{eqnarray*}
which implies our conclusion.
\end{proof}

By applying Theorem \ref{theorem1} and Lemma \ref{lemma1}, we can
prove Theorem \ref{theorembound} as follows.

\vspace{2mm}

\emph{\textbf{Proof of Theorem \ref{theorembound}}}. By applying
Theorem \ref{theorem1} and Lemma \ref{lemma1} directly, we can
obtain
\begin{eqnarray} \label{4.1}
\frac{d}{dt}\lambda_{1}(t)&=&-2\lambda_{1}\kappa(t)+2\int_{M_t}Hg^{im}h_{mq}g^{qj}\nabla_{i}u\nabla_{j}u+2\int_{M_t}uH\nabla_{i}\left(g^{im}h_{mq}g^{qj}\right)\nabla_{j}u\nonumber\\
&=&-2\lambda_{1}\kappa(t)+2\left[\int_{M_t}Hg^{im}\alpha_{m}Hg_{mq}g^{qj}\nabla_{i}u\nabla_{j}u+\int_{M_t}uH\nabla_{i}\left(g^{im}\alpha_{m}Hg_{mq}g^{qj}\right)\nabla_{j}u\right]\nonumber\\
&\geq&-2\lambda_{1}\kappa(t)+2\left(\frac{1}{n}-\epsilon\right)\int_{M_t}H^{2}|\nabla{u}|^{2}+2\int_{M_t}uH\alpha_{j}\nabla_{i}H\cdot{g^{ij}}\cdot\nabla_{j}u.
\end{eqnarray}
On the other hand, by integrating by parts to the last term of the
right hand side of (\ref{4.1}), the pinching condition
(\ref{pinchingc}) and the fact that the first nonzero closed
eigenvalue $\lambda_{1}(t)$ is always positive, we have
\begin{eqnarray} \label{4.2}
&&\int_{M_t}uH\alpha_{j}\nabla_{i}H\cdot{g^{ij}}\cdot\nabla_{j}u=-\frac{1}{2}\left(\int_{M_t}H^{2}\alpha_{j}\nabla_{i}u\cdot{g^{ij}}\cdot\nabla_{j}u+\int_{M_t}uH^{2}\alpha_{j}\cdot{g^{ij}\nabla_{i}\nabla_{j}u}\right) \nonumber\\
&&\qquad \qquad  \geq-\frac{\left(\frac{1}{n}+\epsilon\right)}{2}\int_{M_t}H^{2}|\nabla{u}|^{2}+\frac{\lambda_{1}(t)}{2n}\int_{M_t}u^{2}H^{2}-\frac{1}{2}\int_{M_t}uH^{2}\left(\frac{1}{n}-\alpha_{j}\right)\cdot{g^{ij}\nabla_{i}\nabla_{j}u} \nonumber\\
&&\qquad \qquad
 \geq-\frac{\left(\frac{1}{n}+\epsilon\right)}{2}\int_{M_t}H^{2}|\nabla{u}|^{2}+\frac{\left(\frac{1}{n}-2\epsilon\right)\lambda_{1}(t)}{2}\int_{M_t}u^{2}H^{2}.
\end{eqnarray}
The last inequality holds since, on one hand, for
$0<t<T_{0}<T_{\mathrm{max}}$, we know that $M_{t}$ is strictly
convex (cf. \cite[Corollary 2.5]{lmw}) and bounded, and $H$ is
continuous. Then $H$ has positive maximum and minimum on $M_{t}$,
which are finite. Define
$H_{\mathrm{max}}(t)=\max_{x\in{M_t}}H(x,t)$ and
$H_{\mathrm{min}}(t)=\min_{x\in{M_t}}H(x,t)$ , so
\begin{eqnarray*}
\min_{1\leq{j}\leq{n}}\left|\frac{1}{n}-\alpha_{j}\right|\cdot
H_{\mathrm{min}}^{2}(t)\lambda_{1}(t)\leq\left|\int_{M_{t}}\left(\frac{1}{n}-\alpha_{j}\right)uH^{2}g^{ij}\nabla_{i}\nabla_{j}u\right|\leq\epsilon{H_{\mathrm{max}}^{2}(t)}\lambda_{1}(t).
\end{eqnarray*}
Therefore, by suitably choose $\epsilon$, the equality
\begin{eqnarray*}
-\int_{M_{t}}\left(\frac{1}{n}-\alpha_{j}\right)uH^{2}g^{ij}\nabla_{i}\nabla_{j}u\geq-2\epsilon\lambda_{1}(t)\int_{M_t}u^{2}H^{2}
\end{eqnarray*}
always holds. On the other hand, by Theorem \ref{theoremmain}, we
know that $H_{\mathrm{min}}(t)/H_{\mathrm{max}}(t)\rightarrow1$ as
$t\rightarrow{T_{\mathrm{max}}}$ (this is because $M_{t}$ converges
\emph{spherically} as $t\rightarrow{T_{\mathrm{max}}}$). So, for
sufficiently small $\epsilon>0$, there exists some $\delta>0$ such
that $|H_{\mathrm{min}}(t)/H_{\mathrm{max}}(t)-1|\leq\epsilon$ for
$T_{\mathrm{max}}-\delta\leq{t}<T_{\mathrm{max}}$. This implies that
$|\int_{M_t}u^{2}H^{2}/H_{\mathrm{max}}^{2}(t)-1|$ must be small
enough for $T_{\mathrm{max}}-\delta\leq{t}<T_{\mathrm{max}}$. Hence,
by suitably choose $\epsilon$, we can also get the above inequality.
 Now, substituting (\ref{4.2}) into
(\ref{4.1}) results in
\begin{eqnarray*}
\frac{d}{dt}\lambda_{1}(t)\geq-2\lambda_{1}\kappa(t)+\left(\frac{1}{n}-3\epsilon\right)\int_{M_t}H^{2}|\nabla{u}|^{2}+\left(\frac{1}{n}-2\epsilon\right)\lambda_{1}(t)\int_{M_t}u^{2}H^{2}.
\end{eqnarray*}
Since $\epsilon$ is small enough, without loss of generality, choose
$\epsilon\ll\frac{1}{3n}$, then we have
\begin{eqnarray} \label{4.3}
\frac{d}{dt}\lambda_{1}(t)\geq-2\lambda_{1}\kappa(t).
\end{eqnarray}
 Dividing  both sides of (\ref{4.3}) by $\lambda_{1}$ and
then integrating from $0$ to $t$ ($0<t<T_{\mathrm{m}}$), we have
 \begin{eqnarray*}
 \log\lambda_{1}(t)- \log\lambda_{1}(0)\geq-2\int_{0}^{t}\kappa(\tau)d\tau,
 \end{eqnarray*}
 which implies the assertion of Theorem \ref{theorembound}.
\hfill $\square$

Of course, under the assumption of Lemma \ref{lemma1}, we can also
give a lower bound for the first eigenvalue of the Laplace operator
under the normalized flow (\ref{2.6}) by repeating almost the same
process as above, since from Theorem \ref{theorem1}, we know that
there is no essential difference between the evolution equation of
the first eigenvalue under the unnormalized flow and the
corresponding one under the normalized flow. In fact, we can easily
get
\begin{eqnarray*}
\widetilde{\lambda}_{1}(\tilde{t})\geq
e^{-2\int_{0}^{\tilde{t}}\frac{\widetilde{h}}{n}d\tau}\cdot\widetilde{\lambda}_{1}(0)
\end{eqnarray*}
for $0\leq\tilde{t}<\widetilde{T}_{\mathrm{m}}$.

However, we cannot just repeat the above process to try to get a
similar conclusion for the $p$-Laplace operator when $p\neq2$,
since, as mentioned in Section 3, we do not know whether
$\lambda_{1,p}(t)$ is differentiable or not.

\section{Monotonicity of the first eigenvalues of the Laplacian and the $p$-Laplacian}
\renewcommand{\thesection}{\arabic{section}}
\renewcommand{\theequation}{\thesection.\arabic{equation}}
\setcounter{equation}{0} \setcounter{maintheorem}{0}

By applying Theorems \ref{theorem1} and \ref{theorem2}, we can
easily obtain the following monotonicity for the first eigenvalue.

\begin{theorem} \label{theorem3}
Let $M_{t}$,  $\lambda_{1}(t)$,
$\widetilde{\lambda}_{1}(\tilde{t})$, $\lambda_{1,p}(t)$, and
$\widetilde{\lambda}_{1,p}(\tilde{t})$ be defined as in Theorems
\ref{theorem1} and \ref{theorem2}. Let $T_{\mathrm{m}}$ be defined
by (\ref{key}), and let $\widetilde{T}_{\mathrm{m}}$ be defined as
in Theorem \ref{theorem2}. Denote by $H_{\mathrm{max}}(0)$ and
$H_{\mathrm{min}}(0)$ the maximal and the minimal values of the men
curvature on the initial hypersurface $M_{0}$, respectively. Assume
that $M_{0}$ satisfies the pinching condition (\ref{pinchingc}).
Then we have

(I) If
\begin{eqnarray*}
e^{-2\int_{0}^{t}\kappa(\tau)d\tau}
\left[\frac{H^{-1}_{\mathrm{max}}(0)-2H_{\mathrm{max}}(0)\int_{0}^{t}e^{-2\int_{0}^{\tau}\kappa(s)ds}d\tau}{H_{\mathrm{max}}(0)}\right]^{-1}\leq{n}\kappa(t)
\end{eqnarray*}
for $0\leq{t}<T_{\mathrm{m}}$, then $\lambda_{1}(t)$ is
non-increasing for $0\leq{t}<T_{\mathrm{m}}$ under the flow
(\ref{1.1}), and $\lambda_{1,p}(t)$ is non-increasing and
differentiable almost everywhere for $0\leq{t}<T_{\mathrm{m}}$ under
the flow (\ref{1.1}). If
\begin{eqnarray*}
e^{-2\int_{0}^{t}\kappa(\tau)d\tau}
\left[\frac{H^{-1}_{\mathrm{min}}(0)-2H_{\mathrm{min}}(0)\int_{0}^{t}e^{-2\int_{0}^{\tau}\kappa(s)ds}d\tau}{nH_{\mathrm{min}}(0)}\right]^{-1}\geq{n\kappa(t)}
\end{eqnarray*}
 for
$0\leq{t}<T_{\mathrm{m}}$, then $\lambda_{1}(t)$ is nondecreasing
for $0\leq{t}<T_{\mathrm{m}}$ under the flow (\ref{1.1}), and
$\lambda_{1,p}(t)$ is nondecreasing and differentiable almost
everywhere for $0\leq{t}<T_{\mathrm{m}}$ under the flow (\ref{1.1}).

(II) If
\begin{eqnarray*}
e^{-2\int_{0}^{\tilde{t}}\frac{\widetilde{h}}{n}d\tau}
\left[\frac{H^{-1}_{\mathrm{max}}(0)-2H_{\mathrm{max}}(0)\int_{0}^{\tilde{t}}e^{-2\int_{0}^{\tau}\frac{\widetilde{h}}{n}ds}d\tau}{H_{\mathrm{max}}(0)}\right]^{-1}\leq\widetilde{h}
\end{eqnarray*}
for $0\leq{\tilde{t}}<\widetilde{T}_{\mathrm{m}}$, then
$\widetilde{\lambda}_{1}(\tilde{t})$ is non-increasing under the
normalized flow (\ref{2.6}), and
$\widetilde{\lambda}_{1,p}(\tilde{t})$ is non-increasing and
differentiable almost everywhere under the normalized flow
(\ref{2.6}). If
\begin{eqnarray*}
e^{-2\int_{0}^{\tilde{t}}\frac{\widetilde{h}}{n}d\tau}
\left[\frac{H^{-1}_{\mathrm{min}}(0)-2H_{\mathrm{min}}(0)\int_{0}^{\tilde{t}}e^{-2\int_{0}^{\tau}\frac{\widetilde{h}}{n}ds}d\tau}{nH_{\mathrm{min}}(0)}\right]^{-1}\geq\widetilde{h}
\end{eqnarray*}
 for
$0\leq{\tilde{t}}<\widetilde{T}_{\mathrm{m}}$, then
$\widetilde{\lambda}_{1}(\tilde{t})$ is nondecreasing under the
normalized flow (\ref{2.6}), and
$\widetilde{\lambda}_{1,p}(\tilde{t})$ is nondecreasing and
differentiable almost everywhere under the normalized flow
(\ref{2.6}).
\end{theorem}

\begin{proof}
By (\ref{2.4}) and the fact that the convexity is preserved under
the forced MCF (\ref{1.1}), that is, $M_{t}$ is convex (cf.
\cite[Corollary 2.5]{lmw}), we have
\begin{eqnarray*}
\frac{\partial{H}}{\partial{t}}&=&\Delta{H}+|A|^{2}H-\kappa(t)H\\
&\leq&\Delta{H}+H^{3}-\kappa(t)H.
\end{eqnarray*}
Let $\rho(t)$ be the solution of the initial value problem
\begin{eqnarray*}
\left\{
\begin{array}{lll}
\frac{d}{dt}\rho(t)=\rho^{3}(t)-\kappa(t)\rho(t),\\
\\
\rho(0)=H_{\mathrm{max}}(0):=\max\limits_{x\in{M_0}}H(x,0). & \quad
\end{array}
\right.
\end{eqnarray*}
By applying the maximum principle to the function $H(x,t)-\rho(t)$,
we can obtain
\begin{eqnarray*}
H(x,t)\leq\rho(t)=e^{-\int_{0}^{t}\kappa(\tau)d\tau}
\left[\frac{H^{-1}_{\mathrm{max}}(0)-2H_{\mathrm{max}}(0)\int_{0}^{t}e^{-2\int_{0}^{\tau}\kappa(s)ds}d\tau}{H_{\mathrm{max}}(0)}\right]^{-1/2}.
\end{eqnarray*}
Similarly, by (\ref{2.4}) we have
\begin{eqnarray*}
\frac{\partial{H}}{\partial{t}}&=&\Delta{H}+|A|^{2}H-\kappa(t)H\\
&\geq&\Delta{H}+\frac{H^{3}}{n}-\kappa(t)H.
\end{eqnarray*}
Let $\sigma(t)$ be the solution of the initial value problem
\begin{eqnarray*}
\left\{
\begin{array}{lll}
\frac{d}{dt}\sigma(t)=\sigma^{3}(t)-\kappa(t)\sigma(t),\\
\\
\sigma(0)=H_{\mathrm{min}}(0):=\min\limits_{x\in{M_0}}H(x,0). &
\quad
\end{array}
\right.
\end{eqnarray*}
By applying the maximum principle to the function
$H(x,t)-\sigma(t)$, we can obtain
\begin{eqnarray*}
H(x,t)\geq\sigma(t)=e^{-\int_{0}^{t}\kappa(\tau)d\tau}
\left[\frac{H^{-1}_{\mathrm{min}}(0)-2H_{\mathrm{min}}(0)\int_{0}^{t}e^{-2\int_{0}^{\tau}\kappa(s)ds}d\tau}{nH_{\mathrm{min}}(0)}\right]^{-1/2}.
\end{eqnarray*}
From the proof of Theorem \ref{theorembound}, we know that once the
initial hypersurface $M_0$ satisfies the pinching condition
(\ref{pinchingc}), $M_t$ remains so and
\begin{eqnarray*}
\frac{d}{dt}\lambda_{1}(t)\geq-2\lambda_{1}\kappa(t)+\left(\frac{1}{n}-3\epsilon\right)\int_{M_t}H^{2}|\nabla{u}|^{2}+\left(\frac{1}{n}-2\epsilon\right)\lambda_{1}(t)\int_{M_t}u^{2}H^{2}.
\end{eqnarray*}
Hence, we have
\begin{eqnarray*}
\frac{d}{dt}\lambda_{1}(t)&\geq&-2\lambda_{1}\kappa(t)+\sigma^{2}\left[\left(\frac{1}{n}-3\epsilon\right)\int_{M_t}|\nabla{u}|^{2}+\left(\frac{1}{n}-2\epsilon\right)\lambda_{1}(t)\int_{M_t}u^{2}\right]\nonumber\\
&=&2\lambda_{1}\cdot\left[-\kappa(t)+\left(\frac{1}{n}-\frac{5}{2}\epsilon\right)\sigma^{2}\right],
\end{eqnarray*}
which implies that $\lambda_{1}(t)$ is non-decreasing under the flow
(\ref{1.1}) provided $\sigma^{2}\geq{n\kappa(t)}$.

On the other hand, similar to the proof of Theorem
\ref{theorembound}, one can easily get
\begin{eqnarray*}
\frac{d}{dt}\lambda_{1}(t)\leq-2\lambda_{1}\kappa(t)+2\left(\frac{1}{n}+\epsilon\right)\int_{M_t}H^{2}|\nabla{u}|^{2}+2\int_{M_t}uH\alpha_{j}\nabla_{i}H\cdot{g^{ij}}\cdot\nabla_{j}u.
\end{eqnarray*}
and
\begin{eqnarray*}
\int_{M_t}uH\alpha_{j}\nabla_{i}H\cdot{g^{ij}}\cdot\nabla_{j}u\leq-\frac{\left(\frac{1}{n}-\epsilon\right)}{2}\int_{M_t}H^{2}|\nabla{u}|^{2}+\frac{\left(\frac{1}{n}+2\epsilon\right)\lambda_{1}(t)}{2}\int_{M_t}u^{2}H^{2}.
\end{eqnarray*}
 by applying Theorem \ref{theorem1} and Lemma \ref{lemma1}, and suitably
 choosing $\epsilon$. Combining the above two inequalities yields
\begin{eqnarray*}
\frac{d}{dt}\lambda_{1}(t)&\leq&-2\lambda_{1}\kappa(t)+\rho^{2}\left[\left(\frac{1}{n}+3\epsilon\right)\int_{M_t}|\nabla{u}|^{2}+\left(\frac{1}{n}+2\epsilon\right)\lambda_{1}(t)\int_{M_t}u^{2}\right]\nonumber\\
&=&2\lambda_{1}\cdot\left[-\kappa(t)+\left(\frac{1}{n}+\frac{5}{2}\epsilon\right)\rho^{2}\right],
\end{eqnarray*}
which implies that $\lambda_{1}(t)$ is non-increasing under the flow
(\ref{1.1}) provided $\rho^{2}\leq{n}\kappa(t)$.

Now, for the case of the $p$-Laplacian, by Lemma \ref{lemma1}, if
the evolving hypersurface $M_t$ satisfies (\ref{pinchingc}), $M_t$
remains so. Then, together with Theorem \ref{theorem2} and similar
to the proof of Theorem \ref{theorembound}, at some time
$t_{0}\in[0,T_{\mathrm{m}})$ we have
\begin{eqnarray*}
\frac{d}{dt}\lambda_{1,p}(u,t)&\geq&-p\lambda_{1,p}(t)\kappa(t)+\sigma^{2}\left\{\left[\frac{p-1}{n}-(p+1)\epsilon\right]\int_{M_t}|\nabla{u}|^{p}+\left(\frac{1}{n}-2\epsilon\right)\lambda_{1,p}(t)\int_{M_t}u^{p}\right\}\nonumber\\
&=&p\cdot\lambda_{1,p}(t)\cdot\left[-\kappa(t)+\left(\frac{1}{n}-\frac{p+3}{p}\epsilon\right)\sigma^{2}\right]
\end{eqnarray*}
at the time $t_{0}$, which implies that
\begin{eqnarray*}
\frac{d}{dt}\lambda_{1,p}(u,t)\Big{|}_{t=t_0}\geq0
\end{eqnarray*}
provided $\sigma^{2}\geq{n}\kappa(t)$. Since $\lambda_{1,p}(u,t)$
defined by (\ref{3.6}) is a smooth function with respect to $t$,
then, for any sufficiently small number $\xi>0$, we have
\begin{eqnarray*}
\frac{d}{dt}\lambda_{1,p}(u,t)\geq0
\end{eqnarray*}
on the interval $[t_{0}-\xi,t_{0}]$. Integrating the above
inequality on $[t_{0}-\xi,t_{0}]$, we can obtain
\begin{eqnarray}  \label{5.1}
\lambda_{1,p}(u(\cdot,t_{0}-\xi),t_{0}-\xi)\leq\lambda_{1,p}(u(\cdot,t_{0}),t_{0}).
\end{eqnarray}
By the definition (\ref{3.6}) of $\lambda_{1,p}(u,t)$, we know that
$\lambda_{1,p}(u(\cdot,t_{0}),t_{0})=\lambda_{1,p}(t_{0})$ and
$\lambda_{1,p}(u(\cdot,t_{0}-\xi),t_{0}-\xi)\geq\lambda_{1,p}(t_{0}-\xi)$
at time $t_{0}$. Together this fact with (\ref{5.1}), we have
\begin{eqnarray*}
\lambda_{1,p}(t_{0}-\xi)\leq\lambda_{1,p}(t_{0})
\end{eqnarray*}
for sufficiently small $\xi>0$. It follows that $\lambda_{1,p}(t)$
is monotone non-decreasing under the flow (\ref{1.1}), since $t_{0}$
can be chosen arbitrarily. The fact that $\lambda_{1,p}(t)$ is
differentiable everywhere on $[0,T_{\mathrm{m}})$ can be derived by
applying the classical Lebesgue's theorem. Similarly, if
$\rho^{2}\leq{n}\kappa(t)$, then $\lambda_{1,p}(t)$ is monotone
non-increasing and differentiable everywhere under the flow
(\ref{1.1}). The second assertion (II) of Theorem {\ref{theorem3}}
for the normalized flow can be obtained by almost the same process.
\end{proof}

\begin{remark}\rm{ It is surprising that $\lambda_{1}(t)$ and $\lambda_{1,p}(t)$
have the same monotonicity under the same assumptions, and one may
think that it is not necessary to derive the monotonicity of
$\lambda_{1}(t)$ independently, since $\lambda_{1}(t)$ is only a
special case of $\lambda_{1,p}(t)$, i.e.
$\lambda_{1,p}(t)=\lambda_{1}(t)$ when $p=2$. However, readers can
find that one cannot use the way for proving the monotonicity of
$\lambda_{1}(t)$ to get the monotonicity of $\lambda_{1,p}(t)$
directly (see the proof of Theorem \ref{theorem3} in Section 5).
Besides, by applying Theorem \ref{theoremmain}, we can know more
about $\widetilde{T}_{\mathrm{m}}$. More precisely, we can obtain:
if $\Xi=\infty$ with $\Xi$ defined by (\ref{2.8}), then
$T_{\mathrm{m}}<\infty$, and $\phi(t)\rightarrow\infty$ as
$t\rightarrow{T_{\mathrm{max}}}$, which implies
$\widetilde{T}_{\mathrm{m}}=\int_{0}^{T_{\mathrm{max}}}\phi^{2}(s)ds=\infty$;
if $0<\Xi<\infty$, then $T_{\mathrm{max}}=\infty$, which implies
$\widetilde{T}_{\mathrm{m}}=\int_{0}^{\infty}\phi^{2}(t)dt$; if
$\Xi=0$, then $T_{\mathrm{max}}=\infty$, while
$\widetilde{T}_{\mathrm{m}}=\int_{0}^{T_{\mathrm{m}}}\phi^{2}(t)dt=\int_{0}^{T}\phi^{2}(t)dt<\infty$.
}
\end{remark}

\section*{Acknowledgments}

The author was partially supported by the starting-up research fund
(Grant No. HIT(WH)201320) supplied by Harbin Institute of Technology
(Weihai). This paper was started when the author visited the Beijing
International Center for Mathematical Research (BICMR), Peking
University from March 2013 to April 2013, and the author here is
grateful to Prof. Yu-Guang Shi for the hospitality during his visit
to BICMR. The last revision of this paper was carried out when the
author visited the Chern Institute of Mathematics (CIM), Nankai
University in December 2013, and the author here is grateful to
Prof. Shao-Qiang Deng for the hospitality during his visit to CIM.
The author would like to thank CIM for supplying the financial
support during his visit through the Visiting Scholar Program.


\begin{thebibliography}{99}

\bibitem{ban} B. Andrews, Contraction of convex hypersurfaces in
Euclidean space, \emph{Calc. Var. Partial Differential Equations}
{\bf 2} (1994) 151--171.


\bibitem{eccs} E. Cabezas-Rivas and C. Sinestrai, Volume-preserving flow by powers of the $m$th
mean curvature, \emph{Calc. Var. Partial Differential Equations}
{\bf 38} (2010) 441--469.

\bibitem{caox} X. Cao, Eigenvalues of $(-\Delta+\frac{R}{2})$ on
manifolds with nonnegative curvature operator, \emph{Math. Ann.}
{\bf 337} (2) (2007) 435--441.

\bibitem{caox2} X. Cao, First eigenvalues of geometric operators under the Ricci flow,
\emph{Proc. Amer. Math. Soc.} {\bf 136} (11) (2008) 4075--4078.

\bibitem{xsj} X. Cao, S. Hou and J. Ling, Estimate and monotonicity of the first eigenvalue
under the Ricci flow, \emph{Math. Ann.} {\bf 354} (2) (2012)
451--463.

\bibitem{fmi} P. Freitas, J. Mao and I. Salavessa, Spherical symmetrization and the first eigenvalue of
geodesic disks on manifolds, \emph{Calc. Var. Partial Differential
Equations} (2013) DOI: 10.1007/s00526-013-0692-7.

\bibitem{rsh} R. S. Hamilton, Three-manifolds with positive Ricci curvature, \emph{J. Differential Geometry}
{\bf 17} (1982) 255--306.

\bibitem{Hus} G. Huisken, Flow by mean curvature of convex surfaces into spheres,
\emph{J. Differential Geom.} {\bf 20} (1984) 237--266.

\bibitem{jfli} J.-F. Li, Eigenvalues and energy functionals with monotonicity formulae under Ricci
flow, \emph{Math. Ann.} {\bf 338} (4) (2007) 927--946.

\bibitem{lmw} G. Li, J. Mao and C. Wu, Convex mean curvature flow with a
forcing term in direction of the position vector, \emph{Acta
Mathematica Sinica, English Series} {\bf 28} (2) (2012) 313--332.

\bibitem{lis} G. Li and I. Salavessa, Forced convex mean curvature
flow in Euclidean spaces, \emph{Manuscripta Math.} {\bf 126} (2008)
335--351.

\bibitem{lma} L. Ma, Eigenvalue monotonicity for the Ricci-Hamilton flow, \emph{Ann. Glob. Anal. Geom.}
{\bf 29} (3) (2006) 287--292.

\bibitem{m2} J. Mao, Forced hyperbolic mean curvature flow,
\emph{Kodai Math. J.} {\bf 35} (3) (2012) 500--522.

\bibitem{m3} J. Mao, Deforming two-dimensional graphs in $R^4$ by forced mean curvature flow,
\emph{Kodai Math. J.} {\bf 35} (3) (2012) 523--531.

\bibitem{mlw} J. Mao, G. Li and C. Wu, Entire graphs
under a general flow, \emph{Demonstratio Mathematica}
  {\bf XLII} (2009) 631--640.

\bibitem{m} J. Mao, A class of rotationally symmetric quantum
 layers of dimension 4, \emph{J. Math. Anal. Appl.} {\bf 397} (2) (2013) 791--799.

\bibitem{m1} J. Mao, Eigenvalue inequalities for the $p$-Laplacian on a Riemannian manifold and estimates
for the heat kernel, \emph{J. Math. Pures Appl.} (2013)
DOI:10.1016/j.matpur.2013.06.006.

\bibitem{PhDJingMao} J.\ Mao, Eigenvalue estimation and some
 results on finite topological type, Ph.D.\ thesis, IST-UTL, 2013.

\bibitem{gpere} G. Perelman, The entropy formula for the Ricci flow
and its geometric applications, arXiv:math.DG/0211159.

\bibitem{jr} J. Roth, A remark on almost umbilical hypersurfaces,
\emph{Arch. Math.} (Brno) {\bf 49} (2013) 1--7.

\bibitem{kx1} K. Shiohama and H. Xu, Rigidity and sphere theorems
for submanifolds, \emph{Kyushu J. Math.} {\bf 48} (2) (1994)
291--306.

\bibitem{kx2} K. Shiohama and H. Xu, Rigidity and sphere theorems
for submanifolds II, \emph{Kyushu J. Math.} {\bf 54} (1) (2000)
103--109.

\bibitem{zhaol} L. Zhao, The first eigenvalue of the $p$-Laplace
operator under powers of the $m$th mean curvature flow,
\emph{Results Math.} {\bf 63} (2012) 937--948.



\end{thebibliography}
 \end{document}